\documentclass[12pt]{amsart}
\usepackage{fullpage, times}
\usepackage{amscd,amssymb,amsmath}
\usepackage{graphicx}
\usepackage{color} 

\newtheorem{Proposition}{Proposition}

\newtheorem{Lemma}{Lemma}
\newtheorem{Theorem}{Theorem}

\newtheorem{defn}{Definition}

\newcommand{\proj}{\mathbb{P}}
\newcommand{\pone}{\proj^1}
\newcommand{\ponepone}{\proj^1 \times \proj^1}
\newcommand{\Z}{\mathbb{Z}}

\newcommand{\oh}{{\mathcal{O}}}
\newcommand{\com}{\mathbb{C}}
\newcommand{\Q}{\mathbb{Q}}

\newcommand{\MM}{\mathsf{M}}

\newcommand{\uuu}{{\mathfrak{1}}}
\newcommand{\ppp}{{\mathfrak{p}}}

\newcommand{\lv}{\left |}

\newcommand{\lang}{\left\langle}
\newcommand{\rang}{\right\rangle}

\newcommand{\combinatfactor}{\mathfrak{z}}
\newcommand{\zz}{{\combinatfactor}}

\newcommand{\cF}{\mathcal{F}}
\newcommand{\vac}{v_\emptyset}
\DeclareMathOperator{\Aut}{Aut}

\def\rt{\sqrt{Q}}

\def\rootQ{\sqrt{Q}}
\def\rootQ2{\sqrt{Q_2}}
\def\sl2{\mathfrak{sl}_2}

\def\H{\mathcal{H}}

\begin{document}

\title{A Fock Space Approach to Severi Degrees}
\author{Y. Cooper and R. Pandharipande}
\maketitle

\begin{abstract}
The classical Severi degree counts the number of algebraic curves of fixed genus and class passing through points in a surface.  We express the Severi degrees
of $\proj^1 \times \proj^1$ as matrix elements of the exponential of
a single  operator
$\mathsf{M}_S$
on Fock space.
The formalism puts Severi degrees
on a similar footing as the more developed study of Hurwitz numbers of coverings of curves. The pure genus 1 invariants of the product $E \times \proj^1$
(with $E$ an elliptic curve) are solved via an exact formula for the eigenvalues of $\mathsf{M}_S$
to initial order.
The Severi degrees of $\proj^2$ are also determined by
$\mathsf{M}_S$ via the $\frac{(-1)^{d-1}}{d^2}$ disk multiple cover formula for
Calabi-Yau 3-fold geometries.
\end{abstract}

\section{Introduction}
Let $\proj^2$ be the complex projective plane.
The study of the following  classical problem in enumerative geometry
was initiated in the late $19^{th}$ century by Chasles, Zeuthen, and
Schubert:

\vspace{8pt}
\noindent  {\em How many
algebraic curves in $\proj^2$ of geometric genus $g$ and degree $d$ pass
through $3d+g-1$ general points?}
\vspace{8pt}

\noindent More precisely, plane curves $C\subset \proj^2$ of degree $d$ are parameterized by the projective space $\proj( H^0(\proj^2, \oh_{\proj^2}(d)))$ of dimension $\binom{d+2}{2}-1$.  Let
$r \leq \binom{d-1}{2}$
be a non-negative integer.
The {\em Severi variety},
\begin{equation}\label{sevd}
S^r_d \subset \proj( H^0(\proj^2, \oh_{\proj^2}(d))),
\end{equation}
parameterizing reduced irreducible curves with exactly $r$
nodes (and no other singularities),
is a nonsingular quasi-projective subvariety of codimension $r$, see \cite{ACGH,HM}.  The degree of $S^r_d$ in the embedding \eqref{sevd} is the classical {\em Severi degree} $N^r_d$.

The Severi degree $N^r_d$ can be interpreted as the number of degree $d$ reduced irreducible plane curves with exactly $r$ nodes passing through $$n = \binom{d+2}{2}-1-r$$ general points of $\proj^2$.  Such curves have geometric genus $g=\binom{d-1}{2}-r$. Severi degrees can also be indexed by geometric genus
$$N_{g,d}= N^r_d$$
and be interpreted as the numbers of degree $d$ reduced irreducible plane curves of geometric genus $g$ passing through $n = 3d+g-1$ general points of $\proj^2$.

The modern study of the Severi variety began with the proof of the irreducibility of $S^r_d$ by J. Harris \cite{H}
using the geometry of stable curves in an essential way --- foreshadowing the stable map techiniques which
would be applied years later.
The Severi degrees $N_{g,d}$ of $\proj^2$ were calculated by degeneration techniques
by Z. Ran  \cite{Z}. An elegant treatment
was  given by L. Caporaso and J. Harris \cite{CH}. There are several more recent solutions in Gromov-Witten theory via localization \cite{GP}, relative invariants \cite{IP,LR,L}, and Virasoro constraints \cite{Ga}. Tropical counts also calculate Severi degrees \cite{Mk}.

With all of these methods available, what is there left to ask?  For comparison, consider the enumeration of Hurwitz covers of curves. A very useful approach to Hurwitz counts is via  the matrix elements in Fock space of the exponential of the
operator $\MM_{H}$ associated to the insertion of a 2-cycle.  Since $\MM_H$  is diagonalized by Schur functions, a fundamental connection to the representation theory of the symmetric group is made. The formalism also incorporates, in a basic way, the cut-and-join structure of Hurwitz numbers.  Fock space methods play a crucial role in the study of the associated Toda hierarchy \cite{O,OP,P}. Is there a natural Fock space formalism for Severi degrees?

A direct approach to Severi degrees via the matrix elements in Fock space of a new operator $\MM_S$ is presented here. While the surface $\proj^1 \times \proj^1$ is the most natural to study, the formalism simultaneously captures the Severi degrees of $\proj^2$ blown-up in at most 3 points.  Severi degrees for elliptic fibrations such as $E\times \proj^1$ occur via the trace.  We hope the results presented here will lead to further study of $\mathsf{M}_S$ (and the full algebra of operators discussed in Section \ref{fll4}).

One avenue for further study is the relationship between this approach for curve counting and tropical geometry.  In \cite{BG}, Block and G\"ottsche explore connections between this approach to Severi degrees and tropical curve counting techniques, such as those of Brugall\'e and Mikhalkin \cite{BM}.  In doing so, they extend curve counting via Fock space to a larger class of toric surfaces, those described by an h-transverse lattice polygon.  Furthermore, they extend consideration from Severi degrees to the refined Severi degrees defined and studied by G\"ottsche and Shende and by Block and G\"ottsche.

\section{Hurwitz theory}

\subsection{Fock space}
We start by reviewing the Fock space formalism for the Hurwitz theory of
covers of curves. The results provide direct motivation for our treatment of Severi degrees.

We begin by recalling the definition of the Heisenberg algebra $\H$.  $\H$ is an infinite dimensional Lie algebra.  A basis of $\H$ is given by the operators $1$ and $\{ \alpha_k \}$, where $k \in \Z \setminus \{0\}$. The Lie bracket is given by
\begin{equation} \label{Hcommutation}
\left[\alpha_k,\alpha_l\right] = k \, \delta_{k+l,0}\,.
\end{equation}
$\H$ has an infinite dimensional irreducible representation $\cF$, called the Fock space representation.  $\cF$ contains a distinguished vector called the vacuum vector $\vac$.  As a vector space $\cF$ is freely generated over $\Q$ by the elements
$$\alpha_{-k_1}...\alpha_{-k_n} \vac, \ \ k_i\in\Z_{>0},$$
where by the commutation relations, the order of $\alpha_{-k_1}...\alpha_{-k_n}$ does not matter.

The Lie algebra $\H$ acts on $\cF$ in the following way.  The creation operators
$$\alpha_{-k}, k \in \Z_{>0},$$
act via
$$ \alpha_{-k} (\alpha_{-k_1}...\alpha_{-k_n} \vac) = \alpha_{-k} \alpha_{-k_1}...\alpha_{-k_n} \vac.$$
The annihilation operators
$$\alpha_{k},\ \  k\in\Z_{>0},$$ kill the vacuum,
$$
\alpha_k \cdot \vac =0,\quad k>0 \,,
$$
and their action on all other vectors is determined by the commutation relations (\ref{Hcommutation}).  For example,
$$ \alpha_1 ( \alpha_{-2} \vac ) = \alpha_1 \alpha_{-2} \vac = \alpha_{-2}  \alpha_1 \vac = 0$$
while
$$ \alpha_1 ( \alpha_{-1} \vac ) = \alpha_1 \alpha_{-1} \vac = ([\alpha_1, \alpha_{-1}] + \alpha_{-1} \alpha_1) \vac = \vac.$$

Given a partition $\mu = \mu_1 + ... + \mu_r$, let $Aut(\mu)$ denote the subgroup of $S_r$ which leaves the partition unchanged.  For example, if $\mu = 2 + 3 + 3 + 5 + 5 + 5$, $Aut(\mu) = S_2 \times S_3$.

Let $\zz(\mu)$ denote the combinatorial factor
$$\zz(\mu)= |\Aut(\mu)|\cdot \prod_{i=1}^{\ell(\mu)} \mu_i \ \ .$$
We will work with the natural basis of $\cF$ given by
the vectors
\begin{equation*}
  \lv \mu \rang = \frac{1}{\zz(\mu)} \, \prod_{i=1}^{\ell (\mu )} \alpha_{-\mu_i} \, \vac \,
\end{equation*}
indexed by partitions
$\mu$.
An inner product is defined on $\cF$ by
\begin{equation}
  \label{inner_prod1}
  \lang \mu | \nu \rang =
\frac{u^{-\ell(\mu)}}{\zz(\mu)}\  {\delta_{\mu\nu}} \,.
\end{equation}
The variable $u$ will play the role of the genus parameter.

The operator in Hurwitz theory corresponding to the insertion
of a $2$-cycle is well-known to be given by
\begin{equation}
\label{hur}
\MM_{H}(u) =
\frac12 \sum_{k,l>0}
\Big[ u\  \alpha_{k+l} \, \alpha_{-k} \, \alpha_{-l} +
 \alpha_{-k-l}\,  \alpha_{k} \, \alpha_{l} \Big] \,.
\end{equation}
The operator $\MM_{H}$ is self-adjoint with respect
to the inner product \eqref{inner_prod1}.

\subsection{Hurwitz numbers}

The {\em Hurwitz number} $H_{g,d}^\bullet$ is the automorphism
weighted count of genus $g$, degree $d$ covers of
$\proj^1$ with simple ramification over $2d+2g-2$
fixed points of $\proj^1$. The superscript $\bullet$
indicates the domain of the cover is {\em not} required
to be connected.
If $2d+2g-2 <0$, then $H_{g,d}^\bullet$ vanishes by definition.

The partition function for the Hurwitz numbers of $\proj^1$ is
$$\mathsf{Z}^{\proj^1}=1+\sum_{g\in \mathbb{Z}} u^{g-1}
\sum_{d>0} \  H_{g,d}^{\bullet} \frac{t^{2d+2g-2}}{(2d+2g-2)!} \ Q^d.
$$
Define the vector $\mathsf{v}\in \cF$ by
$$\mathsf{v}= \sum_{d \geq 0} |\ (1^{d})\ \rangle\ .$$
The partition function is expressed in terms
of matrix elements by the formula
$$\mathsf{Z}^{\proj^1} =
\big\langle\ \mathsf{v} \ | \ Q^{|\cdot|}
\exp\big(t\mathsf{M}_H(u)\big) \
\ | \ \mathsf{v}\ \big\rangle.$$
Here, $|\cdot|$ denotes the energy operator with eigenvalue
$|\mu|$ on the basis vector $|\mu\rangle$.
The simplicity of Hurwitz theory is largely due to the
diagonal form of $\MM_{H}$ in the basis of Schur functions.

The partition function $\mathsf{Z}^E$ for Hurwitz covers of
an elliptic curve $E$ is well-known to be given by the
trace,
$$\mathsf{Z}^{E} = \text{tr}\big( Q^{|\cdot|}
\exp\big(t\mathsf{M}_H(u)\big) ).$$
For example, the genus 1 part is given by the $u=0$
specialization,
\begin{equation}\label{b56}
\mathsf{Z}^{E}_1 = \text{tr}\big( Q^{|\cdot|}
\exp\big(t\mathsf{M}_H(0)\big)\big) =
\sum_{\mu} Q^{|\mu|}\
.
\end{equation}
Since $\MM_H(0)$ is nilpotent, all the eigenvalues
of $\MM_H(0)$ vanish and the second equality is
obtained. The sum on the right is over all partitions $\mu$.

\section{Severi theory}
Let $S$ be a nonsingular projective surface.  The moduli space of stable maps
$$\overline{M}_{g,n}^\bullet(S,\beta)$$
from genus $g$, $n$-pointed curves to $S$ representing the class $\beta\in H_2(S,\mathbb{Z})$ has virtual dimension
$$\text{dim}_{\com}\ [\overline{M}^\bullet_{g,n}(S,\beta)]^{vir} \ =\ \int_\beta c_1(S)+ g-1 +n.$$
The superscript $\bullet$ indicates the domain is possibly disconnected, but with no connected components collapsed to points of $S$.
Let
$$\text{ev}_i: \overline{M}^\bullet_{g,n}(S,\beta) \rightarrow S$$
be the evaluation at the $i^{th}$ marked point.
We refer the reader to \cite{FP,KM} for an introduction to stable maps and Gromov-Witten theory.

A Gromov-Witten analogue of the Severi degree is defined by the following
construction. Let
$$n= \int_\beta c_1(S)+g-1$$
be the virtual dimension of the unpointed space $\overline{M}^\bullet_{g}(S,\beta)$.
Let
$$N^{\bullet}_{g,\beta} = \int_{[\overline{M}^\bullet_{g,n}(S,\beta)]^{vir}} \prod_{i=1}^{n} \text{ev}_i^*(\ppp),$$
where $\ppp\in H^4(S,\mathbb{Z})$ is the point class.  If $n<0$, then $N_{g,\beta}^\bullet$ vanishes by definition.

For an arbitrary surface $S$, the Gromov-Witten invariant $N^{\bullet}_{g,\beta}$ may be completely unrelated to the classical Severi degree. Indeed, for Enriques surfaces, the Gromov-Witten invariants are often fractional, and, for $K3$ surfaces, the Gromov-Witten invariants vanish altogether. However, for the surfaces $\proj^2$ and $\proj^1\times \proj^1$ of main interest here, $N^{\bullet}_{g,\beta}$ is well-known to coincide with the (disconnected) classical Severi degree.  In genus 0, the Gromov-Witten invariants of
 $\proj^2$ and $\proj^1\times \proj^1$ are enumerative by 
the convexity of the surfaces \cite{FP}.  In fact,
the stationary Gromov-Witten invariants in all positive genera 
are also enumerative for both surfaces.  While 
excess components in the moduli spaces of stable maps to these surfaces appear
in positive genus, the point conditions can not be satisfied on 
the excess components \cite{GKP}.

When we discuss Severi degrees for a surface, we will {\em always} mean the corresponding Gromov-Witten invariants $N^{\bullet}_{g,\beta}$.
The Severi degrees of  $\proj^1\times \proj^1$ will play a special role.

\section{Fock Space}
To start, we write the cohomology of $\proj^1$ as  the standard direct sum
$$H^*(\proj^1,\Q) = \Q \cdot \uuu \ \oplus\ \Q \cdot \ppp\ $$
where $\uuu$ and $\ppp$ are the unit and point classes respectively.

The Lie algebra $\H [ \proj^1]$ is generated by the operators $1,$ $\{ \alpha_k [\uuu]\},$ and $\{ \alpha_k [\ppp]\}$, where $k \in \Z \setminus \{0\}$. The Lie bracket is given by
\begin{eqnarray} \label{Hp1commutation}
\left[\alpha_k[\uuu],\alpha_l[\ppp]\right] &= &k \, \delta_{k+l,0}\,
\end{eqnarray}
with all other commutators vanishing.

The Fock space $\cF[\proj^1]$ is freely generated over $\Q$ by the elements
$$\alpha_{-k_1}[\uuu]...\alpha_{-k_n}[\uuu] \alpha_{-\ell_1} [\ppp] ... \alpha_{-\ell_m} [\ppp] \vac, \ \ k_i, \ell_j \in\Z_{>0},$$
where by the commutation relations, the order of $\alpha_{-k_1}[\uuu]...\alpha_{-k_n}[\uuu] \alpha_{-\ell_1} [\ppp] ... \alpha_{-\ell_m} [\ppp]$ does not matter.

$\H [\proj^1]$ acts on $\cF[\proj^1]$ in a way analogous to the action of $\H$ on $\cF$.  The creation operators
$$\alpha_{-k} [\uuu] \ \mathrm{and} \ \alpha_{-k} [\ppp], \ k \in \Z_{>0},$$
act via
$$ \alpha_{-k} [\uuu] (\prod_{i,j} \alpha_{-k_i} [\uuu] \alpha_{-\ell_j} [\ppp] \vac) = \alpha_{-k} [\uuu] \prod_{i,j} \alpha_{-k_i} [\uuu] \alpha_{-\ell_j} [\ppp] \vac,$$
and similarly for $\alpha_{-\ell} [\ppp].$
The annihilation operators
$$\alpha_{k} [\uuu] \ \mathrm{and } \ \alpha_{k} [\ppp],\  k\in\Z_{>0},$$ kill the vacuum
$$
\alpha_k [\uuu] ( \vac )=0 \ \mathrm{and} \ \alpha_k [\ppp] ( \vac )=0, \quad k>0 \,,
$$
and their action on any element $\prod_{i,j} \alpha_{-k_i} [\uuu] \alpha_{-\ell_j} [\ppp] \vac$ is determined by the commutation relations (\ref{Hp1commutation}).

A natural basis of $\cF[\proj^1]$ is given by the vectors
\begin{equation*}
  \label{basis}
  \lv \mu,\nu \rang = \frac{1}{\zz(\mu)\zz(\nu)}
 \, \prod_{i=1}^{\ell (\mu )} \alpha_{-\mu_i}[\uuu]
\prod_{j=1}^{\ell (\nu )} \alpha_{-\nu_j}[\ppp]
 \, \vac \,
\end{equation*}
indexed by all pairs of partitions $\mu$ and $\nu$ (of possibly different sizes).  As before, $\zz(\mu)$ denotes the combinatorial factor
$$\zz(\mu)= |\Aut(\mu)|\cdot \prod_{i=1}^{\ell(\mu)} \mu_i \ \ .$$
An inner product is defined by
\begin{equation*}
  \label{inner_prod}
  \lang \mu,\nu | \mu',\nu' \rang =
\frac{u^{-\ell(\mu)}}{\zz(\mu)}
\frac{u^{-\ell(\nu)}}{\zz(\nu)} \delta_{\mu\nu'}\delta_{\nu\mu'}
\,.
\end{equation*}

We define a new operator $\MM_S$ on the Fock space $\cF[\proj^1]$ by the following formula,
$$\MM_S(u,Q) = \sum_{k>0} \alpha_{-k}[\ppp] \alpha_k[\ppp]
+ Q\sum_{|\mu|=|\nu|>0} u^{\ell(\mu)-1} \alpha_{-\mu}[\uuu]\alpha_\nu[\uuu]\ \ .$$
The second sum is over all pairs of nontrivial partitions $\mu$ and $\nu$ of equal size, and
\begin{eqnarray*}
\alpha_{-\mu}[\uuu]&=& \frac{1}{|\Aut(\mu)|}
\prod_{i=1}^{\ell(\mu)} \alpha_{-\mu_i}[\uuu]\ , \\
\alpha_{\nu}[\uuu]&=& \frac{1}{|\Aut(\nu)|}
\prod_{i=1}^{\ell(\nu)} \alpha_{\nu_i}[\uuu]\ . \\
\end{eqnarray*}
As before, $u$ is the genus variable. The new variable $Q$ will be related to the curve class.

The operator $\MM_S$ is self-adjoint with respect to the inner product \eqref{inner_prod} and hence diagonalizable.
The Hurwitz operator $\MM_H$ occurs as a summand in the second term of $\MM_S$.
Unlike $\MM_H$, the operator $\MM_S$ does not appear to be diagonalizable
over $\Q$.

\section{$\proj^1 \times \proj^1$}
Let the variables $Q_1$ and $Q_2$ correspond to the curve classes of the fibers of the first and second projections to $\proj^1$ respectively.  The partition function for Severi degrees of $\proj^1 \times \proj^1$ is
$$\mathsf{Z}^{\proj^1\times \proj^1}=1+\sum_{g\in \mathbb{Z}} u^{g-1}
\sum_{(d_1,d_2)} \  N_{g,(d_1,d_2)}^{\bullet} \frac{t^{2d_1+2d_2+g-1}}{(2d_1+2d_2+g-1)!} \ Q_1^{d_1} Q_2^{d_2}
$$
where the second sum is over all non-negative $d_i$ satisfying $(d_1,d_2)\neq (0,0)$.
The vector
$$\mathsf{v}= \sum_{d_1\geq 0} |\ (1^{d_1}), \emptyset\ \rangle$$
in $\cF[\proj^1]$ plays an important role.

\begin{Theorem} \label{MainTheorem}
${\mathsf Z}^{\proj^1\times \proj^1} = e^{tQ_2/u} \big\langle\ \mathsf{v} \ | \
Q_1^{|\cdot|}
\exp\big(t\mathsf{M}_S(u,Q_2)\big)
\ | \ \mathsf{v}\ \big\rangle.$
\end{Theorem}

Here, $|\cdot|$ is the energy operator on $\cF[\proj^1]$ with eigenvalue $|\mu|+|\nu|$ on the basis vector $|\mu,\nu \rangle$.  The prefactor $e^{tQ_2/u}$ could have been included in the definition of $\mathsf{M}_S$, but we have chosen not to.

\section{Proof of Theorem 1}

\subsection{Overview}
We prove Theorem 1 via the degeneration formula for relative Gromov-Witten invariants \cite{IP,LR,L}.
Consider the Gromov-Witten invariant
 $N_{g,(d_1,d_2)}^{\bullet}$ counting genus $g$ curves through $$n= 2d_1+2d_2+g-1$$ points
on $\proj^1\times \proj^1$.
Let $C$ be a chain of $n+2$ rational curves.  Let $S$ be the surface  $C \times \proj^1$ viewed as
a union of  $n+2$ copies of $\proj^1\times \proj^1$.
We degenerate
\begin{equation}\label{mddx}
\proj^1 \times \proj^1 \leadsto C \times \proj^1
\end{equation}
and distribute the original $n$ point conditions by placing one on each of the middle $n$ components.
We will refer to the $n+2$ components of the degeneration as
$S_{0}, \ldots , S_{n+1}$
and the $n+1$ relative divisors as $D_0,\ldots,D_n$.
The matrix $\mathsf{M_S}$ arises from explicit calculations
on components of the degeneration \eqref{mddx}.

\begin{figure}[h]
\centering
\includegraphics[width=6 in]{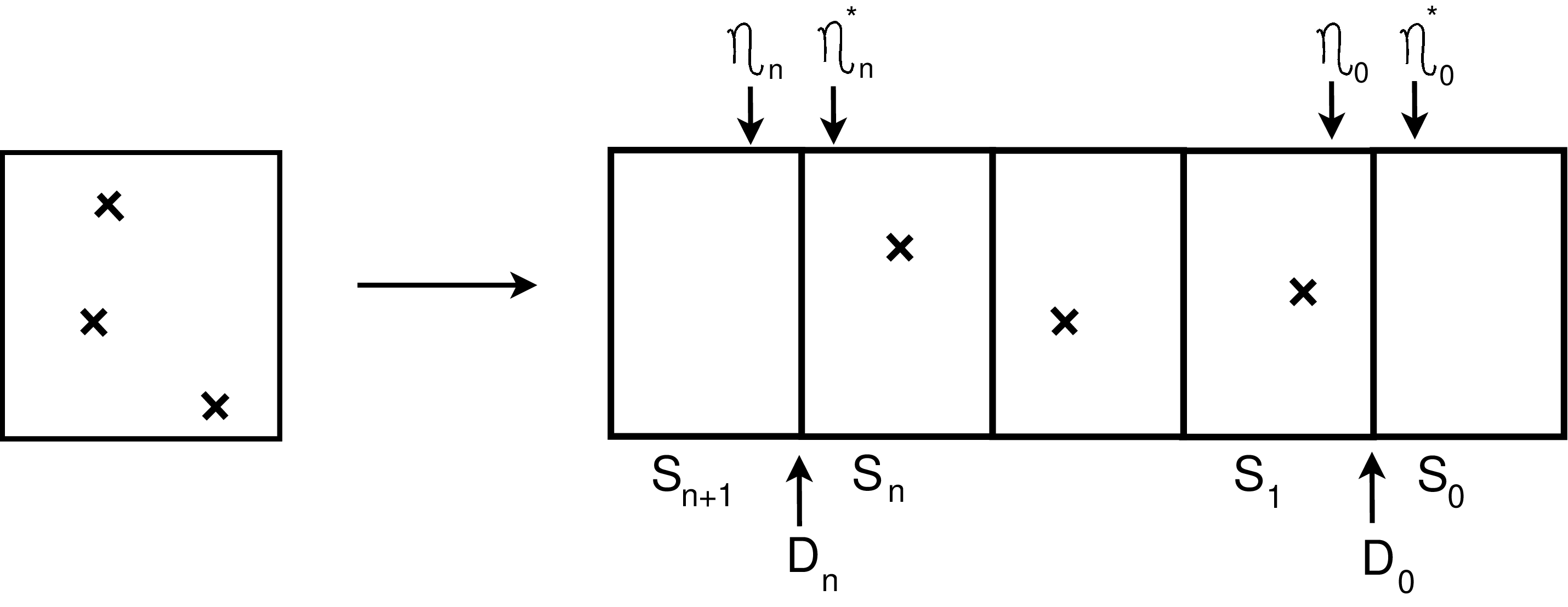}
\caption{The first and last components carry no point conditions, all the rest have one point condition each.  We take $(1,0)$ to be the class of a horizontal line and $(0,1)$ the class of a vertical line.}
\end{figure}

\subsection{Stable relative maps}
A moduli space of stable relative maps is defined for each component $S_i$ of the
above degeneration.
Relative conditions along the divisors $D_i$ are specified
by partitions weighted by
 the cohomology of $\proj^1$.
For $S_i$ where
 $1 \leq i \leq n$, let
$$M^i = \overline{M}^{\bullet}_{g_i,1}(D_i \backslash\, \ponepone / D_{i-1} , (d_1, d_2^i), \eta^*_{i}, \eta_{i-1}, \Gamma^i)$$
denote the moduli space of stable relative degree $(d_1, d_2^i)$ maps of a genus $g_i$ curve to $\ponepone$ of graph type $\Gamma^i$
 satisfying relative
conditions $\eta_{i-1}$ along the divisor $D_{i-1}$ and  $\eta^*_{i}$ along $D_{i}$.  Note the superscript of $\Gamma^i$ and $d_2^i$ is an index, not an exponent.

The graph type $\Gamma^i $ fixes the topology of the map.  Each vertex of $\Gamma^i$
corresponds to a component of the domain curve and is labeled with the genus of that component.
For each relative condition on that domain curve the vertex is given a half-edge labeled
with the corresponding relative condition. The unique marked point is assigned
to a single component of $\Gamma^i$ (which satisfies the incidence condition).

%

The outside components $S_0$ and $S_{n+1}$ play a special role. Following the above
conventions, let
$$M^0 = \overline{M}^{\bullet}_{g_0,0}(D_0 \backslash \ponepone, (d_1, d_2^0), \eta^*_0, \Gamma^0)$$
and let
$$M^{n+1} = \overline{M}^{\bullet}_{g_{n+1},0}( \ponepone / D_{n}, (d_1, d_2^{n+1}), \eta_n, \Gamma^{n+1})\ .$$
For all the above moduli spaces $M^i$, we will view the relative markings  on the domain of the
map
as {\em ordered}.

\subsection{Partition notation}
We take all our partitions to be ordered partitions.
\begin{defn}
Let $\rho$ be a partition of $d$ and let $\rho(k)$ be the number of parts of size $k$ in $\rho$, so $d = \sum_{k = 1}^{\infty} \rho(k) k$.
\end{defn}

Let $\rho = \rho_1 + \cdots + \rho_m$ and $\lambda = \lambda_1 + \cdots + \lambda_n$ be two partitions, and $d = |\rho| + |\lambda|$.  We say
$$\rho [\uuu] + \lambda [\ppp] = \rho_1 [\uuu] + \cdots + \rho_m[\uuu] + \lambda_1 [\ppp]
+ \cdots + \lambda_n [\ppp]$$
is a cohomology weighted partition of $d$, weighted by the cohomology of $\proj^1$.

Let $\cup$ denote concatenation of partitions,
$$\rho \cup \lambda = \rho_1 + \cdots + \rho_m + \lambda_1 + \cdots + \lambda_n.$$

\begin{defn}
Let $\eta = \rho [\uuu] + \lambda [\ppp]$ be a partition weighted by the cohomology of $\proj^1$.  Let
$$m(\eta)= \prod_i \rho_i \prod_j \lambda_j,\ \ \
\text{\em Aut}(\eta) = \text{\em Aut}(\rho) \times \text{\em Aut}(\lambda),\ \ \
\eta^* = \lambda [\uuu] + \rho [\ppp].$$
\end{defn}

\subsection{Degeneration} \label{poneponeDegen}

By the degeneration formula of \cite{IP,LR,L},

\begin{align*}
%
N^\bullet_{g,(d_1,d_2)}
& = \sum_{d^i_2,\eta_i,\Gamma^i}
\left(\int_{[M^{n+1}]} 1 \right) \frac{m(\eta_{n})}{ |\text{Aut}(\eta_n)|} \prod_{i = 1}^{n} \left[ \left(\int_{[M^i]} \text{ev}_1^*(\ppp) \right) \frac{m(\eta_{i-1})}{ |\text{Aut}(\eta_{i-1})|}  \right]
\left(\int_{[M^{0}]} 1 \right)\ .
\end{align*}
The sum is over all degree splittings $$d_2^0 + ... + d_2^{n+1}=d_2,$$
relative conditions $\eta_0, \ldots, \eta_n$, and compatible
graph types $\Gamma^0,\ldots, \Gamma^{n+1}$ which connect to form a genus $g$ curve.
The relative conditions $\eta_i^*$ are set by Definition 2.
On the right side, $[M^i]$ denotes the virtual fundamental class of the moduli
space $M^i$.

Equivalently, we can write the partition function of the Severi degrees of the surface
$\proj^1 \times \proj^1$ as
\begin{multline} \label{RelativeEqn}
 \mathsf {Z}^{\proj^1\times \proj^1} = 1 + \sum_{g,d_1,d_2} Q_1^{d_1} Q_2^{d_2} u^{g-1}\frac{t^n}{n!} \sum_{d_2^i,\eta_i,\Gamma^i}
\left(\int_{[M^{n+1}]} 1 \right) \frac{m(\eta_{n})}{ |\text{Aut}(\eta_n)|} \\
 \times \prod_{i = 1}^{n} \left[ \left(\int_{[M^i]} \text{ev}_1^*(\ppp) \right)
\frac{m(\eta_{i-1})}{ |\text{Aut}(\eta_{i-1})|}  \right]
\left(\int_{[M^{0}]} 1 \right).
\end{multline}
In the above formula, $n=2d_1+2d_2+g-1$ as usual.

\subsection{Geometry of the components}

\subsubsection{The caps}
We  will now analyze the integrals appearing in the degeneration formula \eqref{RelativeEqn}.
Consider first the $i=0$ term corresponding to the component $S_0$.
Let the relative condition be
$$\eta^*_0 = \rho[\uuu] + \lambda [\ppp]$$
where
$\rho$ and $\lambda$ are partitions satisfying $|\rho|+|\lambda| = d_1$.

Let $R$ be a component of the domain curve of a map to $S_0$ parameterized
by $M^0$, and let
$$\sigma[\uuu] + \tau[\ppp]$$
be the relative condition imposed on $R$.  Suppose the genus of $R$ is $h$ and $f_*[R] = \beta = (a,b)$.  The dimension of the space of such maps is
$$\mathrm{dim}_\com \ \overline{M}_{h,0}(\ponepone, \beta)  = \int_{(a,b)} c_1(T_{\ponepone}) + h-1 = 2a + 2b + h - 1.$$
Meanwhile the number of conditions imposed on the map by the relative conditions is
$$\sum (\sigma_i - 1) + \sum \tau_j = a - \ell(\sigma).$$
After setting the dimension to equal the number of conditions,  we obtain
$$a + 2b + h + \ell(\sigma) = 1.$$

Each term on the left hand side is nonnegative, so $b$ must be zero. Certainly $h\leq 1$.
Since $R$ has no marked points to stabilize, $a$ cannot also vanish.  Hence,
the only possible solution is
$$a=1 \ \mathrm{and} \ b=h=\ell(\sigma)=0.$$
We see
 $R$ must be a genus 0 curve mapping with degree 1 onto a line in the class $(1,0)$
 with the relative condition $[\ppp]$.

Therefore the integral over $M^0$ vanishes unless
$$d_2^0 = 0 \ \mathrm{and} \ \eta_0^* = [\ppp] + \cdots + [\ppp]$$
and $\Gamma^0$ is a graph on $d_1$ vertices with a half-edge at each vertex.
If the above conditions are satisfied, the moduli space $M^0$ consists of
a single point which parameterizes a map of $d_1$ disconnected rational curves to $\ponepone$, each with a fixed condition of multiplicity 1, mapping with degree 1 to the unique curve in the class $(1,0)$ passing through that fixed relative condition.
As such a map has no automorphisms, we find
$$\int_{[M^0]} 1 =1 . $$
The analysis for $S_{n+1}$ is identical.

\begin{figure}[h]
  \centering
  \includegraphics[width=3 in]{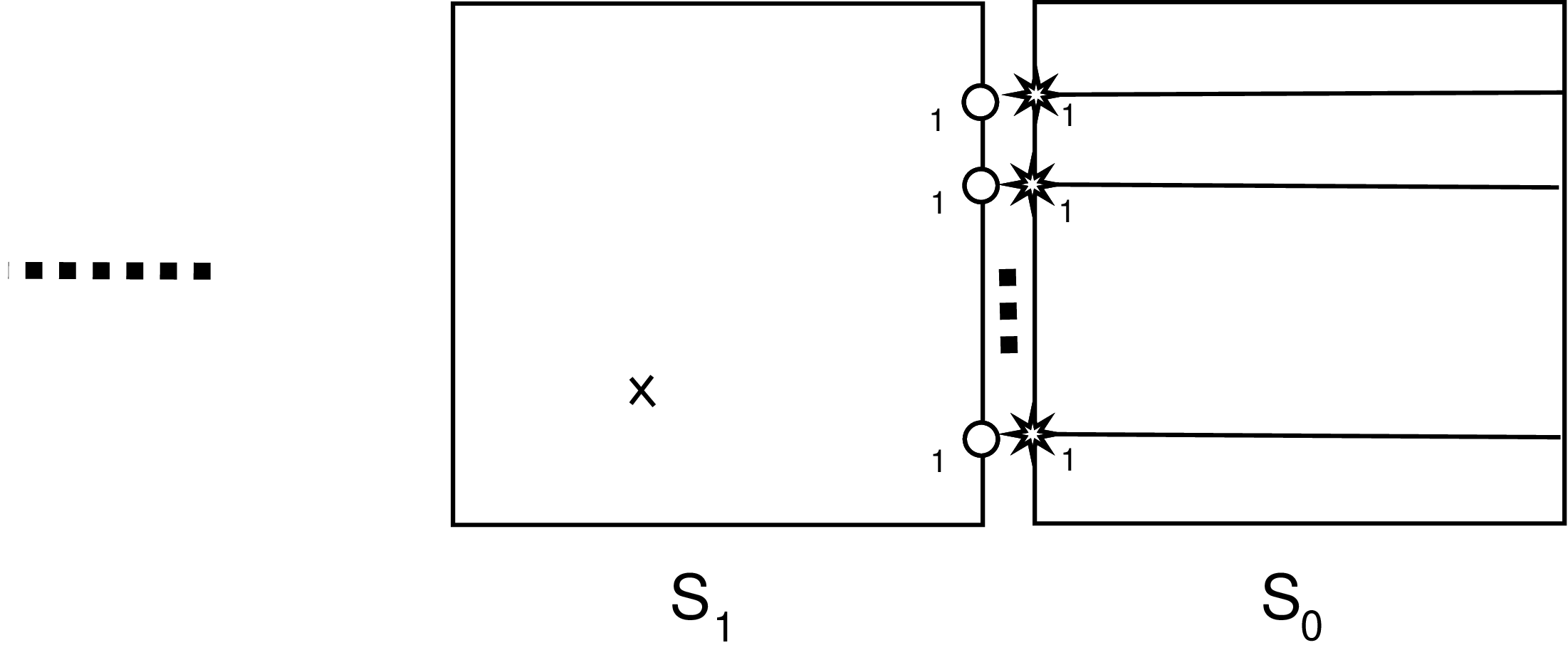}
  \caption{The image of the only map to the cap $S_0$.}
\end{figure}

The evaluations of the integrals in \eqref{RelativeEqn} corresponding
to the end components are

\begin{equation} \label{Capboth}
 \left(\int_{[M^0]} 1 \right) = 1, \ \ \ \ \ \ \ \  \
\
\left(\int_{[M^{n+1}]} 1 \right) \frac{m(\eta_{n})}{|\text{Aut}(\eta_{n})|} = \frac{1}{d_1!}.
\end{equation}

\subsubsection{Middle components} \label{middle11}
Next, we analyze the components  $S_i$ in  case $1 \leq i \leq n$.  Let
$$\eta_{i-1} = \rho[\uuu] + \lambda[\ppp], \ \ \  \eta_i^* = \rho'[\uuu] + \lambda' [\ppp].$$
As before, consider a single genus $h$ connected component $R$ of the domain curve
of  a map to $S_i$ parameterized
by $M^i$.
Let
$$\sigma[\uuu] + \tau[\ppp] \ \ \text{and} \ \ \sigma'[\uuu] + \tau'[\ppp]$$
be the relative conditions imposed on $R$ along $D_{i-1}$ and
$D_{i}$ respectively.
Let $f_*[R] = \beta = (a,b)$.
As $M^i$ is a moduli space of 1-pointed curves, there are two cases:
 either the marked point of lies on $R$ or $R$ is unpointed.

Consider first the case where $R$ does not carry a marked point. Then,
$$\mathrm{dim}_\com \ \overline{M}_{h,0}(\ponepone, \beta)  =  2a + 2b + h - 1\ .$$
The relative conditions impose
$$\sum (\sigma_i-1) + \sum (\sigma'_{i'}-1) + \sum \tau_j  + \sum \tau_{j'} $$
conditions.
After equating the two, we obtain
$$2b + h + \ell(\sigma) + \ell(\sigma') = 1. $$
The unique solution (up to exchanging $\sigma$ and $\sigma'$) is easily determined
to be
$$h = 0,\  b = 0,\ \ell(\sigma) = 1, \ \ell(\sigma') = 0.$$
We see that $\sigma' = \tau = \emptyset$ and also that $\sigma$ and $\tau'$ must each be single part partitions.  We then deduce that $\sigma = \tau' = a$.  
Then  $R$ is rational and maps with degree $|\sigma| = a$ to the line of class $(1,0)$ through the point fixed by $\tau_1' [\ppp]$ ramified totally over $D_{i-1}$ and $D_i$.  The moduli space of such maps is isomorphic as a stack to 
$B \Z/a\Z$.

We consider next the case where $R$ carries the marked point,
$$\mathrm{dim}_\com \ \overline{M}_{h,1}(\ponepone, \beta) = 2a + 2b + h.$$
The relative conditions impose
$$\sum (\sigma_i-1) + \sum (\sigma'_{i'}-1) + \sum \tau_j + \sum \tau_{j'}$$
conditions on such a map.  Setting their difference equal to 2, the degree of $\text{ev}_1^*(\ppp)$, we obtain
$$2b + h + \ell(\sigma) + \ell(\sigma') = 2.$$
Again, the terms on the left hand side are all nonnegative.  The possible solutions are
\begin{eqnarray*} \label{TypeA}
\mathrm{Type \ A:} & & \ \ h= 0,\ b=0,\ \ell(\sigma) = \ell(\sigma') = 1, \\
 \label{TypeBC}
\mathrm{Type \ B:} & &\ \ h = 0,\ b = 1, \ \ell(\sigma) = \ell(\sigma') = 0,\
 \ \ell(\tau), \ell(\tau') \neq 0, \\
 \label{TypeBCd}
\mathrm{Type \ C:} & & \ \ h = 0,\ b = 1, \  \ell(\sigma) = \ell(\sigma') = 0, \ \
\ell(\tau) = \ell(\tau') = 0,
\end{eqnarray*}
The arithmetically allowed solution $h=0$, $b = 0$, $\ell(\sigma) = 2$,
$\ell(\sigma') = 0$ and $h=1$, $b=0$, $\ell(\sigma)=1$, $\ell(\sigma')=0$
do not contribute by simple geometric considerations.

\begin{figure}[h]
  \centering
  \includegraphics[width=6 in]{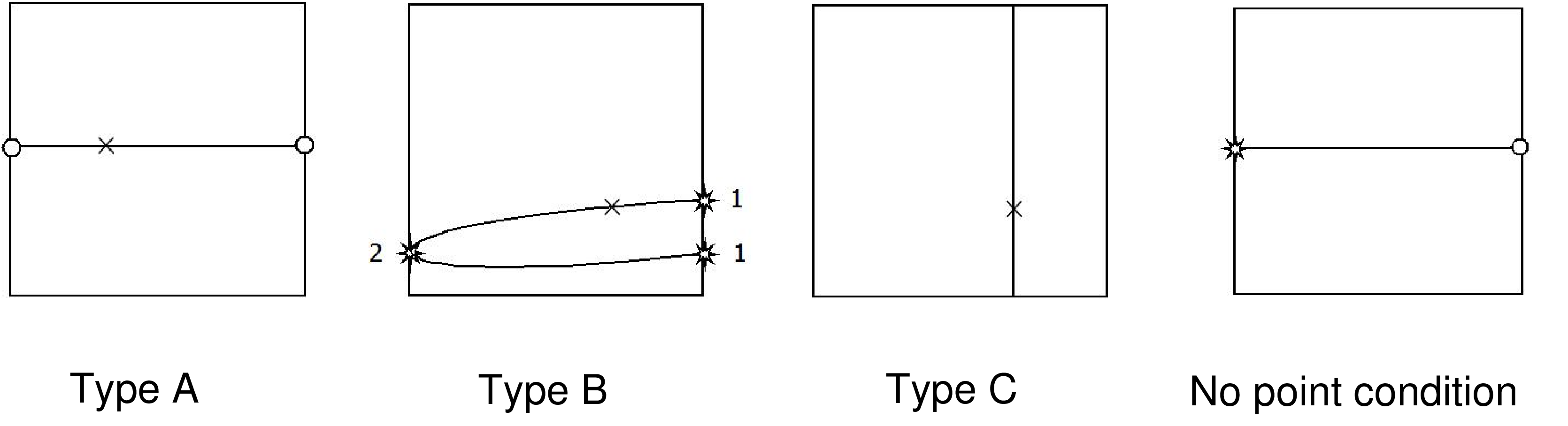}
  \caption{Examples of the images of maps of each type.}
\end{figure}

If $R$ is a component of Type A, then $\sigma = \sigma' = |\sigma|$ is a partition with only one part and $R$ is a 1-pointed rational curve mapping with degree $|\sigma|$ to a line in the class $(1,0)$ totally ramified over the two relative divisors $D_{i-1}$ and $D_i$. The moduli space $M$ of such maps is isomorphic to $\proj^1$.  Since $R$ has a marked point and the map has two ramification points, there are no automorphisms of this map and $\int_{[M]} ev_1^*(\ppp) = 1$.

If $R$ is a component of Type B, then $R$ is a rational curve mapping to $\ponepone$ with image the zero set of
$$x_0 p(y_0, y_1) +  x_1 q(y_0, y_1) $$
where $p$ and $q$ are homogeneous polynomials of degree $|\tau|$.  The map is degree 1 onto its image.  Up to scaling, $q(y_0,y_1)$ is determined by the relative condition $\tau$, and $p(y_0,y_1)$ is determined up to scaling by the relative condition $\tau'$. Again the moduli space $M$ of such maps is isomorphic to $\proj^1$. A general such map has no automorphisms, so $\int_{[M]} ev_1^*(\ppp) = 1$.

If $R$ is a component of Type C, $R$ is a rational curve mapping with degree 1 onto a line in the class $(0,1)$. The moduli space $M$
of such maps is isomorphic to $\proj^1$, and $\int_{[M]} ev_1^*(\ppp) = 1$.

In conclusion, if the partitions $\eta_{i-1}$ and $\eta_i$ are such that the domain curve has $k$ components, then
$$\int_{[M^i]} ev_1^*(\ppp) = \frac{1}{m_1 \cdot ... \cdot m_{k-1}},$$
where $m_1,\ldots,m_{k-1}$ are the degrees with which the $k-1$ unmarked components map to $\ponepone$.

\subsection{Proof of the Main Theorem}

We must show the matrix product
\begin{equation}\label{gfft}
e^{tQ_2/u} \big\langle\ \mathsf{v} \ | \ Q_1^{|\cdot|} \exp\big(t\mathsf{M}_S(u,Q_2)\big) \ | \ \mathsf{v}\ \big\rangle
\end{equation}
is the partition function for Severi degrees of $\ponepone$.
We begin by bringing the exponential inside and rewriting:
\begin{multline*}
e^{tQ_2/u} \big\langle\ \mathsf{v} \ | \ Q_1^{|\cdot|} \exp\big(t\mathsf{M}_S(u,Q_2)\big) \ | \ \mathsf{v}\ \big\rangle\\ =
\big\langle\ \mathsf{v} \ | \ Q_1^{|\cdot|} \exp\left(t \left(\sum_{k>0} \alpha_{-k}[\ppp] \alpha_k[\ppp]
+ Q_2 \sum_{|\mu|=|\nu|>0} u^{\ell(\mu)-1} \alpha_{-\mu}[\uuu] \alpha_\nu[\uuu] + \frac{Q_2}{u}\right) \right) \ | \ \mathsf{v}\ \big\rangle\ .
\end{multline*}
Via the operator
\begin{equation}
\label{jjjj}
 \mathsf{N}_S = \sum_{k>0}  \alpha_{-k}[\ppp] \alpha_k[\ppp] \ + \sum_{|\mu|=|\nu|>0} Q_2 u^{\ell(\mu)-1}  \alpha_{-\mu}[\uuu] \alpha_\nu[\uuu] \ + \frac{ Q_2}{u}\ , \end{equation}
we can rewrite \eqref{gfft} as
\begin{equation}\label{nfft}
\sum_{d_1} Q_1^{d_1} \sum_n \frac{t^n}{n!} \big\langle\ (1^{d_1}), \emptyset \ | \ {\mathsf{N}}_S^n \ |\ (1^{d_1}), \emptyset\ \big\rangle\ .
\end{equation}
It is a simple matter now to match equation \eqref{RelativeEqn} for the partition function
with \eqref{nfft}.
The main point is to match the three summands of $\mathsf{N}$ in \eqref{jjjj} with the three configuration types $A$, $B$, and $C$
of Section \ref{middle11} respectively. We leave the routine bookkeeping to the reader. \qed

\section{Other Surfaces}
\subsection{Overview}
Let $B=Bl_2(\ponepone)$ be the blow-up of $\proj^1\times \proj^1$ at the (distinct) points $p$ and $\widehat{p}$.  Let the variables $E$ and $\widehat{E}$ correspond to the curve classes of the two exceptional divisors.  As before, we define the partition function for Severi degrees as
$$\mathsf{Z}^{B}=1+\sum_{g\in \mathbb{Z}} u^{g-1}
\sum
 N_{g,(d_1,d_2,e,\widehat{e})}^{\bullet}
\frac{t^{2d_1+2d_2 +e+\widehat{e}+g-1}}{(2d_1+2d_2+e+\widehat{e}+
g-1)!} \ Q_1^{d_1} Q_2^{d_2} E^{e} \widehat{E}^{\widehat{e}}
$$
where the second sum is over all tuples $(d_1,d_2,e,\widehat{e})$ corresponding to non-zero curve classes.

We define two vectors $\mathsf{w}, \widehat{\mathsf{w}}$ in $\cF[\proj^1]$. First,
$$\mathsf{w}= e^{\frac{Q_2}{uE}+\frac{E}{u}}\sum_{m\geq 0}\sum_\nu E^{-|\nu|}\prod_{j=1}^{\ell(\nu)}
\frac{(-1)^{\nu_j-1}}{\nu_j}\  |\ (1^{m}), \nu \ \rangle$$
where the second sum is over all partitions $\nu$.
Second,
$$\widehat{\mathsf{w}}= e^{\frac{Q_2}{u\widehat{E}}+ \frac{\widehat{E}}{u}}
\sum_{m\geq 0}\sum_\nu \widehat{E}^{-|\nu|}\prod_{j=1}^{\ell(\nu)}
\frac{(-1)^{\nu_j-1}}{\nu_j}\  |\ (1^{m}), \nu \ \rangle\ .$$

\begin{Theorem} \label{BlowupTheorem}
$
{\mathsf Z}^{B} = e^{tQ_2/u} \big\langle\ \mathsf{w} \ | \
Q_1^{|\cdot|}
\exp\big(t\mathsf{M}_S(u,Q_2)\big)
\ | \ \widehat{\mathsf{w}}\ \big\rangle$.
\end{Theorem}

The proof of Theorem \ref{BlowupTheorem} is given in Section \ref{prf2}. Multiple cover formulas from \cite{Brypan} play a crucial role. The surface $Bl_2(\ponepone)$ is isomorphic to $Bl_3(\proj^2)$. The Severi degrees of the various blow-downs of $Bl_2(\ponepone)$ can be recovered from $\mathsf {Z}^{B}$.  Hence, Theorem \ref{BlowupTheorem} captures the Severi degrees of $\proj^2$ ---  precise formulas are discussed in Section \ref{sec:P2}.

Similarly, for $Bl_1(\ponepone)$, the blow-up of $\proj^1\times \proj^1$ at $p$, we define the partition function for Severi degrees as
$$\mathsf{Z}^{Bl_1(\ponepone)}=1+\sum_{g\in \mathbb{Z}} u^{g-1}
\sum
 N_{g,(d_1,d_2,e,)}^{\bullet}
\frac{t^{2d_1+2d_2 +e+g-1}}{(2d_1+2d_2+e+
g-1)!} \ Q_1^{d_1} Q_2^{d_2} E^{e}
$$
where the second sum is over all tuples $(d_1,d_2,e)$ corresponding to non-zero curve classes.
The proof of Theorem \ref{BlowupTheorem} also yields: 
$$\hspace{.66in} {\mathsf Z}^{Bl_1(\ponepone)} = e^{tQ_2/u} \big\langle\ \mathsf{w} \ | \
Q_1^{|\cdot|}
\exp\big(t\mathsf{M}_S(u,Q_2)\big)
\ | \ \mathsf{v}\ \big\rangle. $$

\subsection{Degeneration}

Consider the degeneration of $\ponepone$ described in Section \ref{poneponeDegen}, with the point $p$ in the rightmost component $S_0$ and the point $\hat{p}$ in the component $S_{n+1}$.  To obtain a degeneration of $B$, we
blow up the points $p$ and $\hat{p}$.

\begin{figure}[h]
\centering
\includegraphics[width=5.5 in]{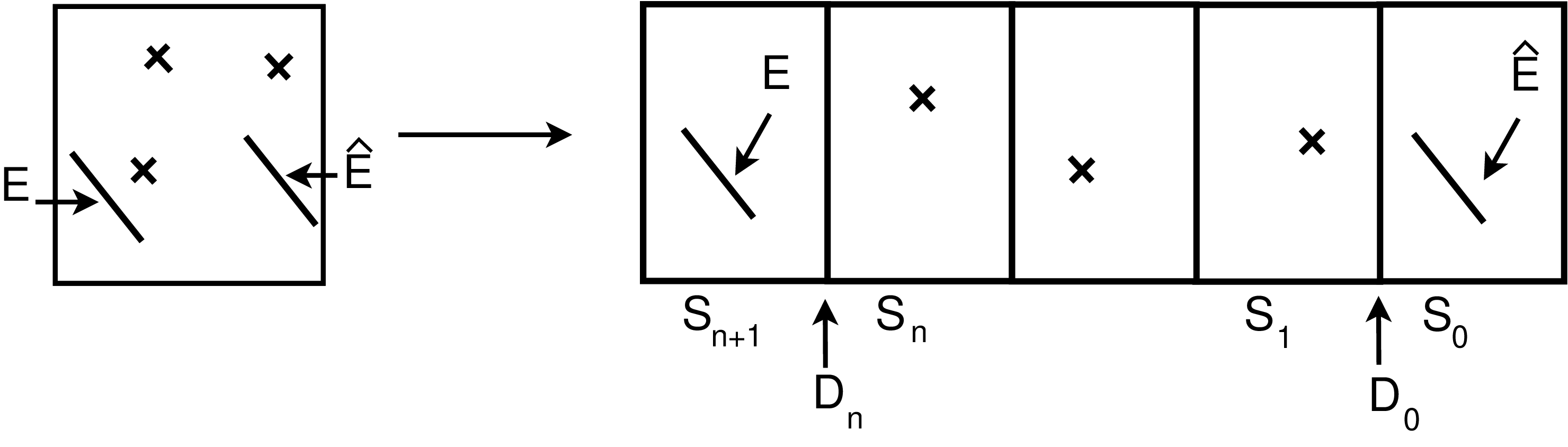}
\caption{Degeneration of $Bl_2(\ponepone)$.  As before, the first and last component carry no point conditions, and all the rest have one point condition each.}
\end{figure}

Again, by the degeneration formula of \cite{IP,LR,L}, we can write the partition function $\mathsf{Z}^B$
for Severi degrees in terms of integrals over moduli spaces of relative maps to the components of the degeneration.
Only the moduli spaces
spaces $M^0$ and $M^{n+1}$ differ from the geometry already considered in the proof Theorem 1.

\subsection{The caps}
 \label{sec:BlowupCaps}
We begin by analyzing the case $i=0$ or $n+1$.  Without loss of generality take $i=0$ and let the relative condition be
$$\eta^*_0 = \rho [\uuu] +\lambda [\ppp].$$
Let $R$ be a component of the domain curve  of a maps to $S_0$ parameterized by $M^0$, and let
$$\sigma[\uuu] + \tau[\ppp]$$
be the relative condition imposed on $R$.  Suppose the genus of $R$ is $h$ and
$$f_*[R] = \beta = a L_1 + b L_2 - e E, $$
where $L_1$ and $L_2$ and the horizontal and vertical line classes.
Since $\beta$ is effective, we must have
$$a+b \geq e, \ \ a \geq 0, \ \mathrm{and} \ b \geq 0.$$

We compute the dimension of the space of maps as
$$\mathrm{dim}_\com \ \overline{M}_{h,0}(Bl_1(\ponepone), \beta)  = \int_{\beta} c_1(T_{Bl_1(\ponepone)}) + h-1 = 2a + 2b - e + h - 1.$$
Meanwhile, the number of conditions imposed on the map by the relative conditions is
$$\sum (\sigma_i - 1) + \sum \tau_j = a - \ell(\sigma).$$
After equating the dimension counts, we obtain
$$\ell(\sigma) + (a+b-e) + b + h = 1.$$

Each term on the left hand side is nonnegative so exactly one term must equal 1, and all others must vanish.  We find there are four types of maps possible.

\vspace{8pt}
\noindent $\bullet \ \ \mathrm{Type \ I:} \ \ell(\sigma) = 1,\ a = e,\ b=0, \ \mathrm{and} \ h=0.$
\vspace{8pt}

Here $R$ is a rational curve mapping as a  degree $a$ cover of the unique curve in the class $L_1 - E$.

\vspace{8pt}
\noindent $\bullet \ \ \mathrm{Type \ II:} \ \ell(\sigma) = 0, \ a = 1,\ e=0,\ b=0, \ \mathrm{and} \ h=0.$
\vspace{8pt}

Here $R$ is a rational curve mapping with degree 1 to a curve in the class $L_1$.

\vspace{8pt}
\noindent $\bullet \ \ \mathrm{Type \ III:} \ \ell(\sigma) = 0,\ a = 0,\ e=-1,\ b=0, \ \mathrm{and} \ h=0.$
\vspace{8pt}

Here $R$ is a rational curve mapping with degree 1 to the exceptional curve $E$.

\vspace{8pt}
\noindent
$\bullet \ \ \mathrm{Type \ IV:} \ \ell(\sigma) = 0,\ a = 0,\ e=1,\ b=1, \ \mathrm{and} \ h=0.$
\vspace{8pt}

Here $R$ is a rational curve mapping with degree 1 to the unique curve in the class $L_2-E$.

\begin{figure}[h]
  \centering
  \includegraphics[width=5 in]{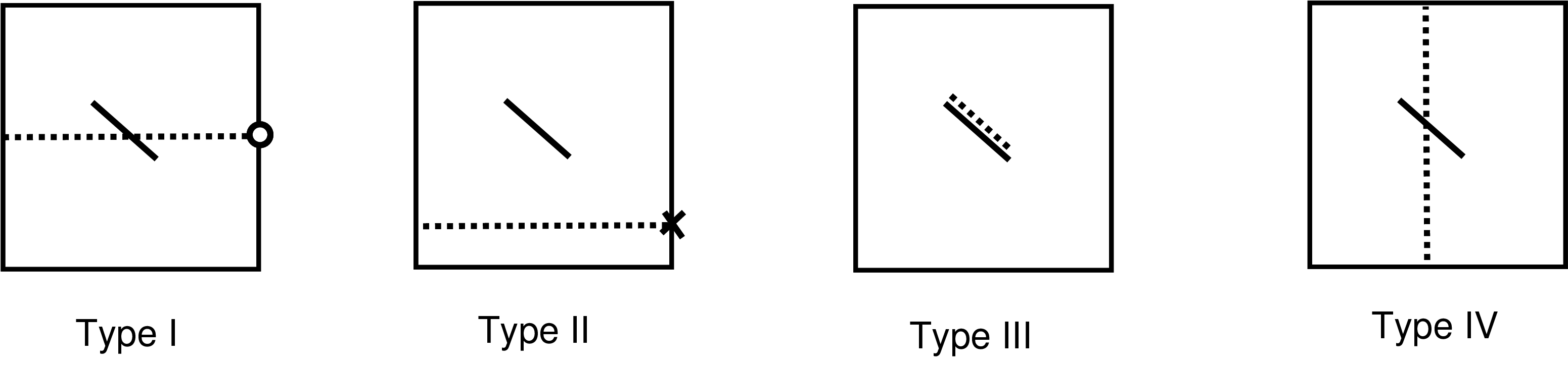}
  \caption{Here are the four possible types of image curves in a cap isomorphic to $Bl_1(\ponepone)$.  The solid line represents the exceptional divisor, the dotted line represents the image curve.  A map of Type I can be a multiple cover of the image curve.  All other maps are
  degree 1.}
\end{figure}

\subsection{Proof of Theorem \ref{BlowupTheorem}} \label{prf2}

\begin{proof}
The geometry of the spaces of maps to $S_1, ..., S_n$ are the same as for the case of $\ponepone$, so we need only compute the contribution of the caps.  Without loss of generality here we consider the cap $S_0$.

By the dimension analysis of Section \ref{sec:BlowupCaps}, only the moduli
space associated to the Types I-IV need be considered.
The moduli spaces $M^0_{II}$, $M^0_{III}$, and $M^0_{IV}$ are simply points.
The geometrically interesting case is
Type I with moduli isomorphic to the space of degree $a$ maps
$$\overline{M}_{0,0}(\proj^1/\infty, a)_{(a)}$$
to $\proj^1$ relative to $\infty\in \proj^1$ and fully ramified along
the relative divisor.
Since the normal bundle of the unique line in the class $L_1-E$
is of degree $-1$, we can write the integral against the
virtual class as
\begin{equation}\label{kdd4}
\int_{[M^0_I]^{vir}} 1 =
\int_{\overline{M}_{0,0}(\proj^1/\infty, a)_{(a)}} c_{\text{top}}\left(
R\pi_*
f^*(\mathcal{O}_{\proj^1}(-1))\right)
\end{equation}
where $\pi$ and $f$ are the standard maps associated to the
universal curve over moduli
$$\pi: \mathcal{C} \rightarrow \overline{M}_{0,0}(\proj^1/\infty, a)_{(a)},
\ \ \ \ f: \mathcal{C} \rightarrow \proj^1\ .$$
Fortunately, the integral \eqref{kdd4} has been calculated in
\cite{Brypan},
\begin{equation*}\int_{\overline{M}_{0,0}(\proj^1/\infty, a)_{(a)}} c_{\text{top}}\left(
R\pi_*
f^*(\mathcal{O}_{\proj^1}(-1))\right) =
 \frac{(-1)^{a - 1}} {a^2}.
\end{equation*}
The rest is bookkeeping.

The formula for $Bl_1(\ponepone)$ can be obtained with the same argument, beginning with the degeneration of $Bl_1(\ponepone)$ to a union of $n+1$ surfaces isomorphic to $\ponepone$ and a cap $S_{n+1}$, isomorphic to $Bl_1(\ponepone)$.  Alternatively, set the coefficients $\widehat{e}$ equal to zero in the formula for ${\mathsf Z}^{B}$.
\end{proof}

\subsection{Severi Degrees of $\proj^2$} \label{sec:P2}
Let $H,F_1,F_2 \in \text{Pic}(Bl_2(\proj^2))$ be the hyperplane class
and the two exceptional classes respectively.
Let $L_1,L_2, E \in \text{Pic}(Bl_1 (\ponepone))$ be
 the pullbacks of the two generators
of the Picard group of $\ponepone$ and the exceptional class.
There is an isomorphism
$$\varphi: Bl_2 (\proj^2) \stackrel{\sim}{\rightarrow} Bl_1 (\ponepone)$$
under which
$$\varphi(H) = L_1 + L_2 -E,\ \ \ \varphi(F_1) = L_1 - E, \ \ \
\varphi(F_2) = L_2 - E\ .$$

The generating function for Severi degrees of $\proj^2$
is
$$\mathsf{Z}^{\proj^2}=1+\sum_{g\in \mathbb{Z}} u^{g-1}
\sum_{d > 0} \  N_{g,d}^{\bullet} \frac{t^{3d+g-1}}{(3d+g-1)!} \ Q^d\ .
$$
The terms of $\mathsf{Z}^{\proj^2}$ are not
a subset of the terms of $\mathsf{Z}^{Bl_2(\proj^2)}$.
For example, the coefficient of $u^{-2} Q^2 t^{4}$ in $\mathsf{Z}^{\proj^2}$ is $3/24$, as there are three pairs of lines through 4 fixed points in $\proj^2$.  The coefficient of $u^{-2} Q^2 t^{4}$ in $\mathsf{Z}^{Bl_2(\proj^2)}$ is $5/24$, as in addition to the three pairs of lines, a smooth conic meeting the divisor $F_1$, union $F_1$, and a smooth conic meeting the divisor $F_2$, union $F_2$, also appear.

However, the \emph{connected} Severi degrees of $\proj^2$ and
$Bl_2(\proj^2)$ are directly related.  For any surface $S$, define
 $$\mathsf{Y}^S= \log \mathsf{Z}^S\ .$$
We view $\mathsf{Y}^S$ as
 the generating function for the connected Severi degrees of $S$.

\begin{Lemma}
The generating function $\mathsf{Y}^{\proj^2}$ appears as a
summand of $\mathsf{Y}^{Bl_2(\proj^2)}.$
\end{Lemma}

\begin{proof}
We must show  the coefficient $N_{g,d}^{\proj^2}$ of $u^{g-1} Q^d t^{n}$ in $\mathsf{Y}^{\proj^2}$ equals the coefficient $N_{g,d}^{Bl_2(\proj^2)}$ of the same term in $\mathsf{Y}^{Bl_2(\proj^2)}$.
Let $p,q\in \proj^2$ be the centers of the blow-up for $Bl_2(\proj^2)$.
The Severi degrees of $\proj^2$
are enumerative.
The nodal solutions are the only stable maps which
pass through general point conditions, and they do {\em not} pass through
$p$ and $q$.
We conclude the strict transforms of the nodal solutions
are the only relevant stable maps
for the connected Severi problem on $Bl_2(\proj^2)$.
\end{proof}

More precisely, under the substitutions
$$Q = \frac{Q_1 Q_2}{E},\  F_1 = \frac{Q_1}{E}, \ \mathrm{and} \ F_2 = \frac{Q_2}{E},$$
we have
$\mathsf{Y}^{Bl_2(\proj^2)}=\mathsf{Y}^{Bl_1(\ponepone)}$.
Therefore computing
$$\mathsf{Z}^{Bl_1 \ponepone} = e^{tQ_2/u} \big\langle\ v \ | \
Q_1^{|\cdot|}
\exp\big(t\mathsf{M}_S(u,Q_2)\big)
\ | \ \widehat{\mathsf{w}}\ \big\rangle$$
via the Fock space formalism and then taking log of the resulting generating function computes the connected Severi degrees of $Bl_2(\proj^2)$, and hence
the connected Severi degrees of $\proj^2$.

\section{The surface $E \times \pone$}

We consider here the Gromov-Witten invariants
 for the surface $E \times \pone$, where $E$ is a smooth curve of genus 1.
By deformation invariance, the Gromov-Witten invariants for $E \times \pone$ are equal to those for $C \times \pone$ where $C$ is a nodal rational curve. The degeneration formula immediately yields the
following result in terms of the trace.

\begin{Theorem}\label{3333}
${\mathsf Z}^{E \times \pone} = e^{tQ_2/u} \ \text{\em tr}\left(
Q_1^{|\cdot|}
\exp\big(t\mathsf{M}_S(u,Q_2)\right)$\ .
\end{Theorem}

A disconnected curve $C$ of genus 1 can have connected components of arbitrary genus (since higher genus
components can be balanced with genus 0 components).
If all of the connected components of $C$ have genus 1,
we call $C$ a disconnected curve of pure genus 1,
and we call the pure genus 1 Gromov-Witten invariants of $E \times \pone$ those which count
 maps from curves of pure genus 1.

The genus 1 invariants are obtained by extracting the $u=0$ coefficient  $\mathsf{Z}^{E \times \pone}$.
The pure genus 1 Gromov-Witten invariants of $E \times \pone$ arise  from the
$u=0$ coefficient of $\text{tr}\left(
Q_1^{|\cdot|}
\exp\big(t\mathsf{M}_S(u,Q_2)\right)$. The prefactor $e^{tQ_2/u}$
in Theorem \ref{3333} contributes
only to impure genus 1 invariants.
Let $\mathsf{Z}_1^{E\times \pone}$ be the partition function for pure genus 1 invariants.
We conclude
$$\mathsf{Z}_1^{E \times \pone} =
\text{tr}\left(
Q_1^{|\cdot|}
\exp\big(t\mathsf{M}_S(0,Q_2)\right)\ .$$
The specialization $\mathsf{M}_S(0,Q)$ has a simple formula,$$
\mathsf{M}_S(0,Q) = \sum_{k>0} \alpha_k[\ppp] \alpha_{-k}[\ppp] + Q\sum_{k>0} \alpha_k[\uuu]\alpha_{-k}[\uuu].$$

\begin{Proposition}
The eigenvalues  of $\mathsf{M}_S(0,Q)$ on the subspace of energy $s$ are
$$\{ \ \ (|\mu|- |\nu|) \sqrt{Q} \ \ \}_{|\mu|+|\nu|=s}\ .$$
\end{Proposition}

\begin{proof}
For convenience, let $\mathsf{M}=\mathsf{M}_S(0,Q)$.
Fix any partition $p$, and let $W_p$ be the subspace spanned by the vectors
$$\{\ |\mu, \nu \rangle\ | \ \mu \cup \nu = p\ \}\ ,$$
corresponding to splittings of $p$ into two subpartitions $\mu$ and $\nu$.
The subspace $W_p$ is invariant under $\mathsf{M}$.  We begin by computing the eigenvalues of
$\mathsf{M}$ on $W_p$.

Let us write $p$ in the frequency representation as
$p = (1^{e_1}, 2^{e_2}, \ldots)$.
Let
$\mu = (1^{l_1}, 2^{l_2}, \ldots)$
and
$\nu = (1^{r_1}, 2^{r_2}, \ldots)$
satisify $\mu \cup \nu = p$, or equivalently, $\l_i + r_i = e_i$.

\begin{multline*}
\mathsf{M} | (1^{l_1}, 2^{l_2},\ldots ) , (1^{r_1}, 2^{r_2},\ldots ) \rangle \\
= \sum_{k}
(k)(l_k) |(1^{l_1}, 2^{l_2},\ldots,k^{l_{k}-1}, (k+1)^{l_{k+1}},\ldots ) , (1^{r_1}, 2^{r_2},\ldots,k^{r_{k}+1}, (k+1)^{r_{k+1}},\ldots) \rangle \\
\ \ \ \ \ \ \ \ + Q (k)(r_k) |(1^{l_1}, 2^{l_2},\ldots,k^{l_{k}+1}, (k+1)^{l_{k+1}},\ldots ) , (1^{r_1}, 2^{r_2},
 \ldots, k^{r_{k}-1}, (k+1)^{r_{k+1}},\ldots) \rangle
\end{multline*}
In fact, $W_p$ is naturally a tensor product
$$W_p = \bigotimes_{k | e_k \neq 0} W_{p_k},$$
where $p_k=(k^{e_k})$ is the subpartition of $p$ consisting of all parts of size $k$.
Let
$$\mathsf{M}_k = \alpha_k[\ppp] \alpha_{-k}[\ppp] + Q\sum_{k>0} \alpha_k[\uuu]\alpha_{-k}[\uuu]\ .$$
$\mathsf{M}_k$ acts nontrivially only on $W_{p_k}$, and we can write $\mathsf{M}= \sum_k \mathsf{M}_k$.

If we choose for each $k$ an eigenvector $v_k \in W_{p_k}$ with eigenvalue $\lambda_k$ for the operator $\mathsf{M}_k$ acting on $W_{p_k}$, then
$\bigotimes_{k | e_k \neq 0} v_k$
is an eigenvector of $\mathsf{M}$ acting on $W_p$  with eigenvalue
$\sum_{k | e_k \neq 0} \lambda_k$.
In the basis of $W_{p_k}$ given by the set of vectors $\{ |\mu, \nu \rangle \}$,
$\mathsf{M}_k = k A_{e_k+1}$, where $A_{e_k+1}$ is the matrix defined in Proposition \ref{BestEigenEverQ}. By
 Proposition \ref{BestEigenEverQ} below, the eigenvalues of $\mathsf{M}_k$ on $W_{p_k}$
are $$\{ k(e_k) \rt,\   k(e_k - 2) \rt,\  \ldots,\  k(-e_k) \rt \}.$$

We conclude  the eigenvalues of $\mathsf{M}$ on $W_p$ are $\{ (|\mu| - |\nu|) \rt \}$
 where $\mu, \nu$ run over all partitions satisfying $\mu + \nu = p$.
Finally, the eigenvalues of $M$ on the subspace of energy $s$ are $\{ (|\mu| - |\nu|) \rt \}$
 where $\mu, \nu$ run over all partitions satisfying $|\mu| + |\nu| = s$.
\end{proof}

\begin{Proposition} \label{BestEigenEverQ}

Let
$A_n =
\begin{bmatrix}
0 & n &        &     &    &\\
Q & 0 &   n-1  &     &    & \\
  & 2Q &      0 & n-2 &   & \\
  &   & \ddots &\ddots &  & \\
  &   &        &    &     & \\
  &   &        &    &     & 1 \\
  &   &        &    & nQ  & 0 \\
\end{bmatrix} \ .
$

\vspace{10pt}
\noindent The eigenvalues of $A_n$ are $\{ n \sqrt{Q}, (n-2)\sqrt{Q}, \ldots , -n \sqrt{Q} \}$.

\end{Proposition}

\begin{proof}
The Lie algebra $\mathfrak{sl}_2$ is generated by $e,f$ and $h$ satisfying the commutation relations $[h,e] = 2e$, $[h,f] = -2f$, and $[e,f] = h$.  Let $x$ and $y$ denote the eigenvectors of $h$ with eigenvalues $1$ and $-1$ respectively.
The standard representation $V$ of $\sl2$ is the 2-dimensional representation given in the basis $\{ x, y \}$ by
$$
e = \begin{bmatrix} 0 & 1 \\ 0 & 0 \end{bmatrix}, f = \begin{bmatrix} 0 & 0 \\ 1 & 0 \end{bmatrix},
h = \begin{bmatrix} 1 & 0 \\ 0 & -1 \end{bmatrix}.
$$
There is one irreducible representation $V_n = \text{Sym}^n(V)$ of dimension $n+1$ for each $n \geq 0$.
The action of $\sl2$ on $V$ induces an action on each $V_n$, and the basis $\{ x,y \}$ of $V$ gives a natural basis $\{ x^n, x^{n-1} y, \ldots , y^n\}$ of $V_n$.
These are eigenvectors for $h$ with eigenvalues $\{n, n-2, ..., -n\}$ respectively.
In the above basis, the induced action of $e$ on $V_n$ is
$$
\begin{bmatrix}
0 & n &        &     &    &\\
  & 0 &   n-1  &     &    & \\
  &   &      0 & n-2 &    & \\
  &   & \ddots &\ddots &  & \\
  &   &        &    &     & \\
  &   &        &    &     & 1      \\
  &   &        &    &     & 0      \\
\end{bmatrix}\ ,
$$
and the action of $f$ on $V_n$ is
$$
\begin{bmatrix}
0 &   &        &     &    &\\
1 & 0 &     &     &    & \\
  & 2 &      0 &   &    & \\
  &   & \ddots &\ddots &  & \\
  &   &        &    &     & \\
  &   &        &    &     &       \\
  &   &        &    &  n  & 0      \\
\end{bmatrix}
$$
In the chosen basis, $e + Qf$ equals the matrix $A_n$.

The operator $e+Qf$ acts on $V$ as
$$
\begin{bmatrix}
0 & 1 \\
Q & 0 \\
\end{bmatrix}
$$
with eigenvectors $z = x+ \rt y$ and $w = x- \rt y$ with eigenvalue  $\rt$ and $-\rt$ respectively.
The basis $\{z,w\}$ of $V$ induces the natural basis $\{ z^n, z^{n-1}w, ..., w^n \}$ for
$V_n = \text{Sym}^n(V)$. In this  basis,
$$
e+Qf =
\begin{bmatrix}
n \rt &     &        & \\
  & (n-2) \rt &        & \\
  &     & \ddots & \\
  &     &        & (-n) \rt \\
\end{bmatrix}.
$$
We conclude that the eigenvalues of $A_n$ are $\{ n \rt , (n-2) \rt ,
\ldots, -n \rt \}$, as desired.

\end{proof}

As a consequence of Proposition 1, we find
\begin{equation}\label{bb34}
\mathsf{Z}^{E\times \pone}_1 =
\text{tr}\big(
Q_1^{|\cdot|}
\exp\big(t\mathsf{M}_S(0,Q_2)\big)\big)=
\sum_{\mu,\nu}
Q_1^{|\mu|+|\nu|}
e^{(|\mu|-|\nu|)t \sqrt{Q_2}}
\end{equation}
where the sum is over all pairs of partitions $\mu$ and $\nu$
(of possibly different sizes).
Formula \eqref{bb34}
is the Severi analogue of \eqref{b56}.

\section{Further directions}
\subsection{Rationality}
A rationality result holds for the following generating functions of
Severi degrees of $\proj^1 \times \proj^1$. Let
$$\mathsf{R}_{0}=1+ \sum_{g \in \Z} u^{g-1} \sum_{d_2 >0} \  N_{g,(0,d_2)}^{\bullet}
 \  Q_2^{d_2},$$
and, for $a>0$, let
$$\mathsf{R}_{a}=\sum_{g \in \Z} u^{g-1} \sum_{d_2\geq 0} \  N_{g,(a,d_2)}^{\bullet}  \ Q_1^{a} Q_2^{d_2}\ .
$$

\begin{Proposition}
For each $a \geq 0$, we have
$$Q_1^{-a} \mathsf{R}_a \in \mathbb{Q}(u, Q_2).$$
\end{Proposition}

\begin{proof}
Recall $\mathsf{N}_S$ defined in \eqref{jjjj}.
By Theorem \ref{MainTheorem},
$$
Q_1^{-a} {\mathsf{R}}_a= \big\langle\ 1^{a},\emptyset \ | \
\big(1+\mathsf{N}_S(u,Q_2)+\mathsf{N}^2_S(u,Q_2)+...\big)
\ | \ 1^{a},\emptyset \big\rangle.
$$
The right hand side is $\frac{1}{a!}$ times the coefficient of $| \ \emptyset,1^{a} \big\rangle$ in
$$
\big(1-\mathsf{N}_S(u,Q_2)\big)^{-1}
\ | \ 1^{a},\emptyset \big\rangle \ .
$$
The inverse of $1-\mathsf{N}_S(u,Q_2)$ can be computed
by the cofactor expansion.
%
\end{proof}

\subsection{Further operators} \label{fll4}
A point condition in the Severi problem corresponds to the operator
$\MM_S(Q,u)$ on $\mathcal{F}[\mathbb{P}^1]$.
In fact, an algebra of commuting operators containing $\MM_S(Q,u)$
is determined by appropriate descendent and relative insertions.
Commutation holds for geometric reasons:
when we degenerate we can assign conditions to the components of
the degeneration in any order.

\begin{figure}[h]
\centering
\includegraphics[width=6 in]{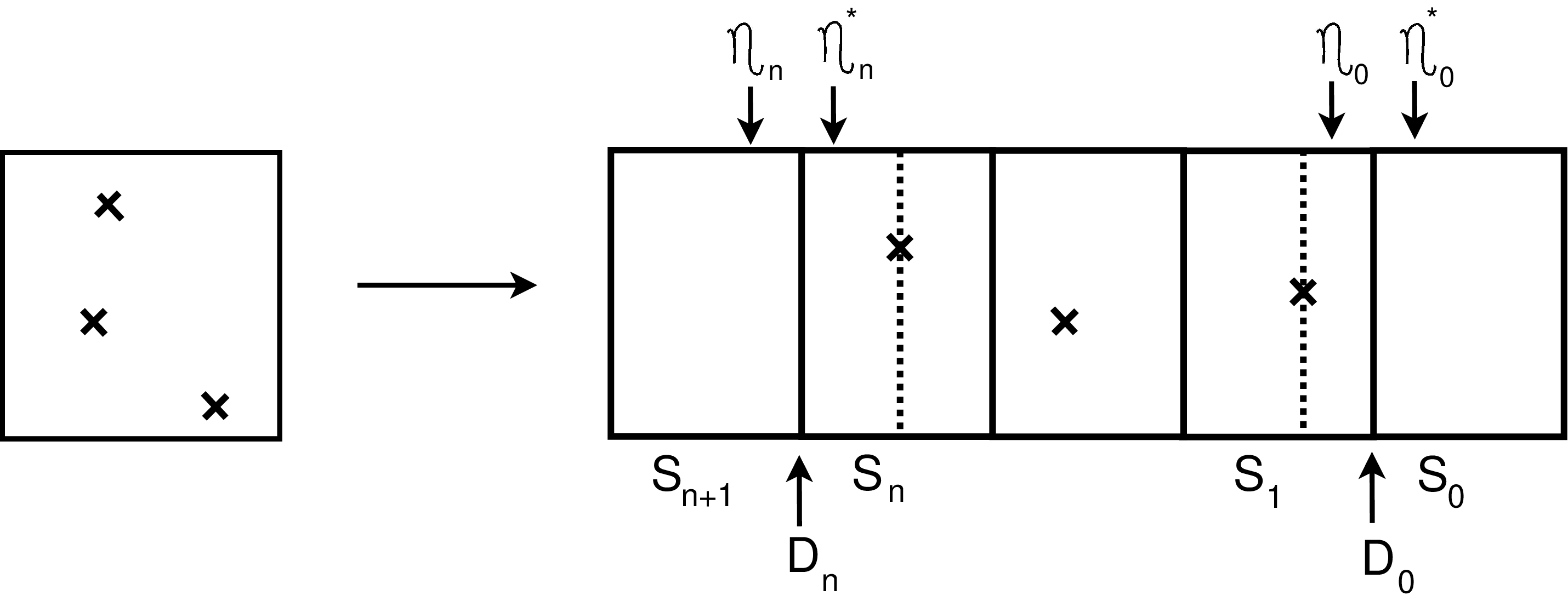}
\caption{We can insert tangency conditions, as in the second and fourth components here, where we require the map to vanish to second order along the line $(0,1)$ at a fixed point.}
\end{figure}

As an example, we consider the relative condition defined by
a single tangency to a vertical line (in class $(0,1)$)
at a fixed point.
The corresponding operator $\MM_F(Q,u)$ on the Fock space $\cF[\proj^1]$ is given by the following formula,
\begin{eqnarray*}
\MM_F(Q,u) &=& \sum_{k=|\mu|, \ell(\mu)=2} \alpha_{-k}[\ppp] \alpha_{\mu}[\ppp] + u \alpha_{-\mu}[\ppp] \alpha_{k}[\ppp]\\
&+& \sum_{|\mu|=|\nu \cup k |, \ k>0,\ \ell(\mu)+\ell(\nu)\geq 2} Q u^{\ell(\mu)-1}
\alpha_{-\mu}[\uuu] \alpha_{\nu}[\uuu] \alpha_k[\ppp]\\
&+&
\sum_{|\mu \cup k|=|\nu|,\ k>0,\  \ell(\mu)+\ell(\nu)\geq 2}
Q u^{\ell(\mu)}
\alpha_{-\mu}[\uuu] \alpha_{-k}[\ppp] \alpha_{\nu}[\uuu] \\
&+& \sum_{k>1} \frac{k^2-1}{12}u \alpha_{-k}[\ppp]\alpha_{k}[\ppp] \ .
\end{eqnarray*}
Here, $\nu$ can be the empty partition
in the second sum and $\mu$ can be empty in the
third sum.

We have computed $\MM_F(u,Q)$, analogous to $\MM_S(u,Q)+\frac{Q}{u}$, by
an analysis similar to the computation of Section \ref{middle11}.  In each component,
 we analyze the geometries corresponding to solutions of the following equation
$$2b + g + \ell(\sigma) + \ell(\sigma') = 3.$$
The most interesting coefficient $\frac{k^2-1}{12}$
arises from a genus 1 contribution but is determined from the
easier genus 0 contributions and the commutation
 $$[\MM_S(Q,u), \MM_F(Q,u)] = 0\ .$$
The operators
 corresponding to higher order tangencies will involve increasingly complicated
 terms. The commutation relation is a considerable
constraint.
We leave further exploration of the algebra of descendent
and relative operators for future research.

\section{Acknowledgements}
We thank M. Maydanskiy and A. Miller for helpful discussions.
Much of the work was done during visits of Y.C. to
IST Lisbon and
ETH Z\"urich.
R.P. was partially supported by the Swiss National Science
Foundation grant SNF 200021143274.

\vspace{+8 pt}
\noindent Department of Mathematics\\
\noindent Princeton University\\
\noindent Washington Rd\\
\noindent Princeton, NJ 08544\\
\noindent USA\\
\noindent yaim@math.princeton.edu\\

\vspace{+8 pt}
\noindent Departement Mathematik\\
\noindent ETH Z\"urich\\
\noindent R\:{a}mistrasse 101\\
\noindent 8092 Zurich\\
\noindent Switzerland\\
\noindent rahul@math.ethz.ch \\

\end{document}